\documentclass[12pt]{amsart}
\usepackage{amsmath,amssymb,latexsym}
\usepackage{fullpage}

\theoremstyle{plain}
\newtheorem{theorem}{Theorem}
\newtheorem{proposition}{Proposition}
\newtheorem{lemma}{Lemma}
\newtheorem{definition}{Definition}

\def\A{\mathbb{A}}

\def\Ind{\text{Ind}}
\def\Mat{\text{Mat}}

\begin{document}

\title{Criteria for the Existence of Cuspidal Theta Representations}
\author{Solomon Friedberg}
\author{David Ginzburg}
\address{Department of Mathematics, Boston College, Chestnut Hill, MA 02467-3806, USA}
\address{School of Mathematical Sciences, Tel Aviv University, Ramat Aviv, Tel Aviv 6997801,
Israel}
\thanks{This work was supported by the US-Israel Binational Science Foundation,
grant number 2012019, and by the National Security Agency, grant number
H98230-13-1-0246 and the National Science Foundation, grant number 1500977 (Friedberg).}
\begin{abstract} 
Theta representations appear globally as the residues of Eisenstein series on covers of groups; their
unramified
local constituents may be characterized as subquotients of certain principal series.   A cuspidal theta
representation is one which is equal to the local twisted theta representation at almost all places.
Cuspidal theta representations are known to exist but only for covers of $GL_j$, $j\leq 3$.
In this paper we establish necessary conditions for the existence of cuspidal theta representations on the $r$-fold metaplectic
cover of the general linear group of arbitrary rank.  
\end{abstract}

\maketitle

\section{Introduction and Main Results}\label{setup}

Let $r\ge2$, let $F$ be a number field containing a full set of $r$-th roots of unity $\mu_r$, and let $\A$ denote
the adeles of $F$.  For $n\geq2$, 
let $GL_n^{(r)}({\A})$ denote an $r$-fold cover of the general linear group. This
group is a cover of $GL_n(\A)$ with fibers given by $\mu_r$ and multiplication defined by a certain two-cocycle $\sigma$.
The group $GL_n^{(r)}(\A)$ is
obtained by piecing together local metaplectic groups $GL_n^{(r)}(F_\nu)$ over the places $\nu$ of $F$ (the group $GL_n^{(r)}(F_\nu)$
is, however, not the $F_\nu$-points of an algebraic group).  Following Takeda \cite{Tak},
we shall use
the local cocycle give by Banks-Levy-Sepanski \cite{B-L-S}, which is block-compatible, and adjust it by a 
coboundary to construct a global cocycle. 

Let  $\Theta_n^{(r)}$ denote the theta representation
on the group $GL_n^{(r)}({\A})$.  This representation was defined in Kazhdan-Patterson  \cite{K-P} 
using the residues of  Eisenstein series,
as follows.  (Kazhdan and Patterson work with a different cocycle than Takeda but the groups are isomorphic.)
Let $B_n$ be the standard Borel subgroup of $GL_n$, and $T_n\subseteq B_n$ denote the maximal torus of $GL_n$.
Let $\mathbf{s}$ be a multi-complex variable, and define the character 
$\mu_{{\mathbf s}}$ of $T_n(\A)$ by $\mu_{{\mathbf s}}(\text{diag}(a_1,\ldots,a_n))=\prod_i |a_i|^{s_i}$.
If $H$ is an algebraic subgroup of $GL_n$, let $H^{(r)}(F_\nu)$ (resp.\ $H^{(r)}(\A)$) denote the full inverse image
of $H(F_\nu)$ (resp.\ $H(\A)$) in  $GL_n^{(r)}(F_\nu)$ (resp.\ $GL_n^{(r)}(\A)$).
Let $Z(T_n^{(r)}(\A))$ denote the center of $T_n^{(r)}(\A)$.
Then $\mu_{{\mathbf s}}$ uniquely determines a genuine character of $Z(T_n^{(r)}(\A))$. Choose a maximal abelian
subgroup $A$ of  $T_n^{(r)}(\A)$, extend this
character to a character of $A$, and induce it to $T_n^{(r)}(\A).$ Then extend trivially to $B_n^{(r)}(\A)$ using the canonical 
projection from $B_n^{(r)}(\A)$ to $T_n^{(r)}(\A)$, and further induce it to the group $GL_n^{(r)}({\A})$. 
We write this induced representation $\Ind_{B_n^{(r)}(\A)}^{GL_n^{(r)}({\A})}\mu_{{\mathbf s}}$.  It follows from \cite{K-P} that this construction is independent of the choice of $A$ and of the extension of characters. Forming the Eisenstein series $E(g,{\mathbf s})$ attached to this induced representation, it follows from \cite{K-P}, that when $\mu_{{\mathbf s}}=\delta_{B_n}^{\frac{r+1}{2r}}$
(with $\delta_{B_n}$ the modular function of $B_n$), this Eisenstein series has a nonzero residue representation.  This is the representation $\Theta_n^{(r)}$.

Let $\nu$ be a finite place for $F$ such that $|r|_\nu=1$. Defining similar groups over the local field $F_\nu$, it
follows from \cite{K-P} that the local induced representation $\Ind_{B_n^{(r)}(F_\nu)}^{GL_n^{(r)}(F_\nu)}\delta_{B_n}^{\frac{r+1}{2r}}$ has a unique unramified subquotient which we again denote by
$\Theta_n^{(r)}$. This representation is also the unique unramified subrepresentation of
$\Ind_{B_n^{(r)}(F_\nu)}^{GL_n^{(r)}(F_\nu)}\delta_{B_n}^{\frac{r-1}{2r}}$.
If  $\chi_\nu$ denotes an unramified character of $F_\nu^\times$, then the induced representation 
$\Ind_{B_n^{(r)}(F_\nu)}^{GL_n^{(r)}(F_\nu)}\chi_\nu^{\frac{1}{r}} \delta_{B_n}^{\frac{r-1}{2r}}$
is also reducible, and one can define the local twisted theta representation $\Theta_{n,\chi_\nu}^{(r)}$ as the unique
unramified sub-representation. Here $\chi_\nu^{\frac{1}{r}}$ is defined on the group $Z(T_n^{(r)}(F_\nu))$ to be the character such that $\chi_\nu^{\frac{1}{r}}((\text{diag}(a_1^r, \ldots, a_n^r),\zeta))=\zeta \,\chi_\nu(a_1\ldots a_n)$. 

Returning to the global case, we have
\begin{definition}\label{def1}
An automorphic representation $\pi$ of $GL_n^{(r)}({\A})$ is called a theta representation if for almost all places $\nu$ there are unramified characters $\chi_\nu$ such that the unramified constituent of $\pi$ is equal to $\Theta_{n,\chi_\nu}^{(r)}$. If $\pi$ is cuspidal, we say that $\pi$ is a cuspidal theta representation.
\end{definition}
The interesting cases of such theta representations are when the local characters $\chi_\nu$ are the unramified characters of a global automorphic character $\chi$. We shall write $\Theta_{n,\chi}^{(r)}$ for such a representation.

Examples of such representations may be constructed as follows.  Suppose that $\chi=\chi_1^r$ for some global character $\chi_1$. Then one can construct theta representations $\Theta_{n,\chi}^{(r)}$ by means of residues of Eisenstein series, by \cite{K-P}. (The case $\chi=1$ was described above.) However, these representations are never cuspidal.

In Flicker \cite{F}, a classification of all theta representations for the covering groups $GL_2^{(r)}({\A})$, $r\geq2$,
with $c=0$ in the sense of \cite{K-P} is given
using the trace formula. The case $n=r=2$ was also studied by Gelbart and Piatetski-Shapiro \cite{G-PS}.  When $n=r=3$, 
Patterson and Piatetski-Shapiro \cite{P-PS} constructed a cuspidal theta representation $\Theta_{n,\chi}^{(r)}$ 
for any $\chi$ which is not of the form $\chi_1^3$, again for the cover with $c=0$. 
This construction applied the converse theorem. No other examples of such representations are known.

The basic problem is then to understand for what values of $r$ and $n$, and for what characters $\chi$, there exists a cuspidal theta representation $\Theta_{n,\chi}^{(r)}$. We shall give a necessary condition for the existence of such
a representation.  However, we
do not determine whether or not these conditions are sufficient.

First,  if $r<n$ such cuspidal representations do not exist. This follows trivially since every cuspidal 
automorphic representation of $GL_n^{(r)}({\A})$ must be generic, but the local unramified  representation $\Theta_{n,\chi_\nu}^{(r)}$ is not generic if $r<n$. Hence we may assume that $r\ge n$. Our main result is

\begin{theorem}\label{th1}
Fix a natural number $r$, and an automorphic character $\chi$ of $GL_1({\A})$. Then there is at most one natural number $n$ such that there is a nonzero cuspidal theta representation $\Theta_{n,\chi}^{(r)}$. Moreover, if such $n$ exists, then $n$ 
divides $r$. If a cuspidal theta representation $\Theta_{n,\chi}^{(r)}$ exists for some $n$ which divides $r$, then $\chi\ne\chi_1^r$ for any character $\chi_1$.
\end{theorem}

For $n=2$ this result follows from \cite{F}. In \cite{F}, a character $\chi$ such that $\chi\ne\chi_1^r$ is called an odd character for the number $r$.

To establish the Theorem, we need to prove three things. First that $n$ divides $r$. We prove this in Section~\ref{divide}. 
Second, the uniqueness property of the number $n$. We prove this in Section~\ref{unique}.
Then in Section~\ref{char} we prove the condition on the character $\chi$.   The basic tool in these sections is the 
study of Eisenstein series obtained by inducing copies of cuspidal theta 
representations, and their residues.  There is a unipotent orbit attached to the automorphic representation generated by these 
residues, and we determine this. 
However, if any of the conditions of the Theorem are violated then this leads to a contradiction.  For example, if there are
two cuspidal theta representations  ${\Theta}_{m,\chi}^{(r)}$ and 
${\Theta}_{n,\chi}^{(r)}$ attached to the same character $\chi$ with $m<n$, a suitable Eisenstein series obtained by mixing them
in the inducing data has a residue. The Fourier coefficient of this residue attached to its
unipotent orbit can be analyzed in two different ways, with contradictory vanishing properties.  The
contradictory properties are due to 
a lack of symmetry for this representation under the outer automorphism of the Dynkin diagram, which takes the relevant parabolic
used in making the induction to its
associated parabolic, in terms of the constant terms that it supports.

We will establish Theorem~\ref{th1} using residues of Eisenstein series, but parts of it follows easily if one accepts Conjecture~1.2 in Bump and Friedberg
\cite{B-F}. Indeed, suppose that for some character $\chi$, there is a cuspidal theta 
representation $\Theta_{r,\chi}^{(r)}$ defined on the group $GL_r^{(r)}({\A})$. Assume that for some $m<r$, one can define a theta representation $\Theta_{m,\chi}^{(r)}$ which need not be cuspidal, but corresponding to the same character $\chi$. Then, assuming Conjecture~1.2 in \cite{B-F}, the following identity follows from  \cite{B-F}, Proposition~2.1
\begin{equation}\label{ran1}
\int\limits_{GL_m(F)\backslash GL_m({\A})} \overline{\theta}_{r,\chi}^{(r)}\begin{pmatrix} g&\\ &I_{r-m}\end{pmatrix}\theta_{m,\chi}^{(r)}(g)\, |\text{det} g|^{s-\frac{r-m}{2}}\,dg= Z_S(\chi,s)\,L^S(rs- \frac{r-1}{2},\chi^{-1}\otimes\Theta_{m,\chi}^{(r)}).
\end{equation}
Here ${\theta}_{r,\chi}^{(r)}$ is a vector in the space of ${\Theta}_{r,\chi}^{(r)}$, and similarly for ${\theta}_{m,\chi}^{(r)}$. Also, $S$ is a set of places, including the archimedean places,  such that outside of $S$ all data is unramified. Finally, $L^S$ is the partial $L$-function interpreted as in \cite{B-F}, and $Z_S(\chi,s)$ is a product of local integrals defined on the places in $S$. 

Now from the definition of the partial $L$-function, it follows that this term
contributes a finite product of partial zeta functions to the right-hand-side of \ref{ran1}. Hence for suitable $s$ the term $L^S(rs- \frac{r-1}{2},\chi^{-1}\otimes\Theta_{m,\chi}^{(r)})$ has a simple pole.
Since the integrals involved in $Z_S(\chi,s)$ are all Whittaker type integrals, it is not hard to prove that given any complex number $s$, there is a choice of data such that $Z_S(\chi,s)$ is not zero at $s$. Hence, for suitable $s$ and suitable data,
the right hand side of \eqref{ran1} has a simple pole. But the left hand side of \eqref{ran1} is holomorphic for all $s$ since ${\Theta}_{r,\chi}^{(r)}$ is cuspidal. This is a contradiction, and hence for all $m<r$ the group $GL_m^{(r)}({\A})$ has no cuspidal theta representation associated with $\chi$. Moreover, if $\chi=\chi_1^r$ for some character $\chi_1$, then for any $m<r$, we can consider the theta representation ${\Theta}_{m,\chi}^{(r)}$ as constructed in \cite{K-P}. Since the left-hand side of \eqref{ran1} still represents a holomorphic function even if ${\Theta}_{m,\chi}^{(r)}$ is not cuspidal while the right-hand side has a pole, we once again derive a contradiction. Hence $\chi\ne\chi_1^r$.

We thank Erez Lapid for helpful conversations.

\section{Residues of Eisenstein Series}\label{residue}

Given $l$ natural numbers $n_1\ge n_2\ge\ldots\ge n_l>0$, let ${\Theta}_{n_i,\chi}^{(r)}$ denote theta representations attached to a fixed character $\chi$. Let $k=n_1+\cdots +n_l$, so $\lambda:=(n_1,n_2,\ldots ,n_l)$ is a partition of $k$.
Let $P_{n_1,\ldots,n_l}$ be the standard parabolic subgroup of $GL_k$ whose Levi part $M_{n_1,\ldots,n_l}$ is $GL_{n_1}\times\cdots\times GL_{n_l}$ embedded diagonally
\begin{equation*}
(g_1,g_2,\ldots,g_l)\mapsto \text{diag}(g_1,g_2,\ldots ,g_l)\quad :\quad g_j\in GL_{n_j},
\end{equation*} and let $U_{n_1,\ldots,n_l}$ denote the unipotent radical of $P_{n_1,\ldots,n_l}$.

Let ${\mathbf s}=(s_1,\dots,s_{l})$ be a multiple complex variable.  Then one may form an Eisenstein series $E_{\lambda,\chi}^{(r)}(g,{\mathbf s})$ 
on the group $GL_k^{(r)}({\A})$ attached to the representations 
$({\Theta}_{n_1,\chi}^{(r)},{\Theta}_{n_2,\chi}^{(r)},\cdots,{\Theta}_{n_l,\chi}^{(r)})$ by a variant of
standard parabolic induction. Once one has a representation of $M^{(r)}_{n_1,\ldots,n_l}(\A)$ the construction
is the standard `averaging' one (see, for example, M\oe glin-Waldspurger \cite{M-W}, II.1.5);
we frequently suppress the dependence of this series on the test vector used
in the averaging from the notation for the Eisenstein series.
However, since the inverse images of the groups
$GL_{n_i}(\A)$ in $M^{(r)}_{n_1,\ldots,n_l}(\A)$ do not commute,  one must restrict to a smaller subgroup and then
induce or extend from that.  Let 
$$GL_{n_j,0}(\A)=\{g\in GL_{n_j}(\A) \mid \det g \in (\A^\times)^r F^\times\}.$$  
Then the inverse images of these
groups in $M^{(r)}_{n_1,\ldots,n_l}(\A)$ commute,
and the group $S$ that they generate is thus isomorphic to the fibered direct product of the $GL^{(r)}_{n_j,0}(\A)$
over $\mu_r$.  Accordingly, one first restricts each representation 
$\Theta_{n_i,\chi}^{(r)}|\text{det}(\cdot)|^{s_i}$ to $GL_{n_i,0}^{(r)}(\A)$, and takes the usual tensor product to obtain a genuine representation of $S$. 
One may then proceed to extend this representation to $M^{(r)}_{n_1,\ldots,n_l}(\A)$,
either by extending functions by zero (as in Suzuki  \cite{S}, Section 8), by 
inducing
to $M^{(r)}_{n_1,\ldots,n_l}(\A)$ (as in Brubaker and Friedberg \cite{Br-Fr},
though that paper is written in the language of $S$-integers as a substitute for the adeles),  or by first extending to a larger subgroup of 
$M^{(r)}_{n_1,\ldots,n_l}(\A)$, under certain hypotheses, and then inducing
from that subgroup to $M^{(r)}_{n_1,\ldots,n_l}(\A)$ (as in Takeda \cite{Tak}).  Our main computations will take
place in the subgroup generated by $S$ and by unipotent subgroups, which split via the trivial section 
$u\mapsto (u,1)$, hence any of these (slightly different) foundations are sufficient for the arguments given here.
We will sometimes abuse the notation (as we already did in the case of induction from the Borel subgroup)
and describe  $E_{\lambda,\chi}^{(r)}(g,{\mathbf s})$ as the Eisenstein series attached to the 
induced representation
\begin{equation}\label{ind1}
\Ind_{P_{n_1,\ldots,n_l}^{(r)}(\A)}^{GL_k^{(r)}({\A})}({\Theta}_{n_1,\chi}^{(r)}|\text{det}(\cdot)|^{s_1}\otimes {\Theta}_{n_2,\chi}^{(r)}|\text{det}(\cdot)|^{s_2}\otimes\cdots\otimes {\Theta}_{n_l,\chi}^{(r)}|\text{det}(\cdot)|^{s_l}).
\end{equation}
As a representation, this induced space is the vector space spanned by the residues $E_{\lambda,\chi}^{(r)}(g,{\mathbf s})$
as one varies over all test vectors.

The Eisenstein series $E_{\lambda,\chi}^{(r)}(g,{\mathbf s})$
has a simple pole, similarly to the case $n_i=1$ for all $i$ which is described in Section~\ref{setup}.
Indeed, by Definition~\ref{def1}, the unramified constituent at a place $\nu$ of the representation ${\Theta}_{n_i,\chi}^{(r)}$ is a quotient of 
$\Ind_{B_{n_i}^{(r)}(F_\nu)}^{GL_{n_i}^{(r)}(F_\nu)}\chi_\nu^{\frac{1}{r}} \delta_{B_{n_i}}^{\frac{r+1}{2r}}$, where
$\chi=\prod_\nu \chi_\nu$. This means that the unramified constituent of the induced representation \eqref{ind1} is an induced representation of the form
$\Ind_{B_k^{(r)}(F_\nu)}^{GL_k^{(r)}(F_\nu)}\chi_\nu^{\frac{1}{r}} \mu_{{\mathbf s}}$ where $\mu_{{\mathbf s}}$ is a genuine character of the group $Z(T^{(r)}_k(F_\nu))$ defined as follows. Let $t=\text{diag}(A_1,\ldots, A_l)$ where each $A_i$ is a diagonal matrix in $GL_{n_i}(F_\nu)$ which consists of $r$-th powers. Let $\tilde{t}=(t,\zeta)\in Z(T^{(r)}_k(F_\nu))$. Then we define
$$\mu_{{\mathbf s}}(\tilde{t})=\zeta\,\delta_{B_{n_1}}^{\frac{r+1} {2r}}(A_1)\ldots \delta_{B_{n_l}}^{\frac{r+1} {2r}}(A_l)\prod_i |A_i|^{s_i}.$$
Arguing as in \cite{K-P}, one sees that the Eisenstein series $E_{\lambda,\chi}^{(r)}(g,{\mathbf s})$ has a simple pole at the point $\mu_{{\mathbf s}}= \delta_{B_k}^{\frac{r+1} {2r}}$.

We remark that the existence of this pole does not depend on whether some of the representations ${\Theta}_{n_i,\chi}^{(r)}$ are cuspidal or not. 
A similar construction with all of the representations being cuspidal was studied by
Suzuki \cite{S}, Sections 8 and 9. In that reference, the author also assumes that the Shimura lifts of the 
cuspidal representations in question are also cuspidal. In our case this does not happen, but the 
argument about the existence of the pole is the same.

Let ${\mathcal L}_{k,\lambda,\chi}^{(r)}$ denote the residue representation of the above Eisenstein series at the above point. 
Then the construction of the representation 
${\mathcal L}_{k,\lambda,\chi}^{(r)}$ is inductive in the following
sense.  For $1\leq j\leq l$ let $\lambda_j$ be a partition of $n_j$.
Form the representations ${\mathcal L}_{n_j,\lambda_j,\chi}^{(r)}$.
Then 
we can form the Eisenstein series attached to the representations
$(\eta_{{\mathbf s}}{\mathcal L}_{n_1,\lambda_1,\chi}^{(r)}, \eta_{{\mathbf s}}{\mathcal L}_{n_2,\lambda_2,\chi}^{(r)},\cdots, 
\eta_{{\mathbf s}}{\mathcal L}_{n_l,\lambda_l,\chi}^{(r)})
$
where $\eta_{{\mathbf s}}$ is an unramified character of $P_{n_1,\ldots,n_l}$. We shall denote this Eisenstein
series by $E_{\lambda_1,\ldots,\lambda_l,\chi}^{(r)}(g,{\mathbf s})$. 
As in the above, and also as in Section~\ref{setup} we deduce that this Eisenstein series has a simple pole at $\eta_{{\mathbf s}}=
\delta_{P_{n_1,\ldots,n_l}}^{\frac{r+1}{2r}}$, and the representation
generated by the residues is ${\mathcal L}_{k,\lambda,\chi}^{(r)}$.

\section{The Divisibility Condition}\label{divide}

Suppose that ${\Theta}_{n,\chi}^{(r)}$ is a cuspidal theta representation defined on $GL_n^{(r)}({\A})$
and that $n$ does {\sl not}
divide $r$. We shall derive a contradiction.  First, we construct the residue representation
${\mathcal L}_{nl,\lambda,\chi}^{(r)}$ on $GL_{nl}^{(r)}({\A})$ where $l$ is a natural
number and $\lambda=(n^l)$.  For convenience we sometimes omit $\lambda$ from the notation,
writing ${\mathcal L}_{nl,\chi}^{(r)}$ instead of ${\mathcal L}_{nl,\lambda,\chi}^{(r)}$.  Thus
${\mathcal L}_{n,\chi}^{(r)}={\Theta}_{n,\chi}^{(r)}$. 

In general, if $\varphi$ is an automorphic function on a group $H({\A})$ and $U$ is any 
unipotent subgroup of $H$, we write $\varphi^U$ for the constant term of $\varphi$ along $U$ 
\begin{equation}\label{uni1}
\varphi^U(h)=\int\limits_{U(F)\backslash U({\A})}\varphi(uh)\,du.\notag
\end{equation}
Also, if $\psi_U$ is a character of $U(F)\backslash U({\A})$, we write
\begin{equation}\label{uni11}
\varphi^{U,\psi_U}(h)=\int\limits_{U(F)\backslash U({\A})}\varphi(uh)\,\psi_U(u)\,du.\notag
\end{equation} 
We shall be concerned with the case that 
$U=U_{(l-m)n,mn}$, the unipotent radical of the maximal parabolic subgroup of $GL_{nl}$ whose Levi part is $GL_{(l-m)n}\times GL_{mn}$,
with $1\le m<l$.

We start with the following 
\begin{proposition}\label{prop1}
Fix $m$, $1\le m<l$,  let
$P=P_{(l-m)n,mn}$ and $U=U_{(l-m)n,mn}$.

(i)~
Let $\varphi_{nl,\chi}^{(r)}$ be a function in the space of ${\mathcal L}_{nl,\chi}^{(r)}$.  Then there are functions $\varphi_{(l-m)n,\chi}^{(r)}$ in
the space of ${\mathcal L}_{(l-m)n,\chi}^{(r)}$ and $\varphi_{mn,\chi}^{(r)}$ 
in the space of ${\mathcal L}_{mn,\chi}^{(r)}$ such that
\begin{equation}\label{type11}
{(\varphi_{nl,\chi}^{(r)}})^U(t(v_1,v_2))=\delta_{P}^{\frac{r-1}{2r}}(t)
\varphi_{(l-m)n,\chi}^{(r)}(v_1)\,\varphi_{mn,\chi}^{(r)}(v_2)
\end{equation}
for all unipotent elements $v_1\in GL_{(l-m)n}(\A)$ and 
$v_2\in GL_{mn}(\A)$ and all $t$ which are $r$-th powers and in the center of the Levi subgroup of $P$. 

(ii)~
For $i=1,2$, let $\psi_{V_i}$ be characters of $V_i(F)\backslash
V_i({\A})$. Then the integral 
\begin{equation*} 
\int\limits_{V_1(F)\backslash V_1({\A})}\int\limits_{V_2(F)\backslash
V_2({\A})}{(\varphi_{nl,\chi}^{(r)}})^U((v_1,v_2))\,\psi_{V_1}(v_1)\,\psi_{V_2}(v_2)\,dv_1\,dv_2
\end{equation*}
is zero for all $\varphi_{nl,\chi}^{(r)}$ in the space of ${\mathcal L}_{nl,\chi}^{(r)}$
if  $(\varphi_{(l-m)n,\chi}^{(r)})^{V_1,\psi_{V_1}}$ or 
$(\varphi_{mn,\chi}^{(r)})^{V_2,\psi_{V_2}}$ is zero for all functions
$\varphi_{(l-m)n,\chi}^{(r)}$ in
the space of ${\mathcal L}_{(l-m)n,\chi}^{(r)}$ or all $\varphi_{mn,\chi}^{(r)}$ 
in the space of ${\mathcal L}_{mn,\chi}^{(r)}$.
\end{proposition}

When $r=1$ and ${\Theta}_{n,\chi}^{(r)}$ is a
cuspidal representation of $GL_n({\A})$, a similar statement is given in Offen-Sayag \cite{O-S}, Lemma~2.4
and Jiang-Liu \cite{J-L}, Lemma~4.2.

\begin{proof}
The proof is based on a standard argument using unfolding of the
Eisenstein series, and closely follows \cite{O-S}, \cite{J-L}, and \cite{M-W} II.1.7. We sketch it briefly. 
Let $E_{nl,\chi}^{(r)}(g,s)$ denote the Eisenstein series attached to the induced representation
$$\Ind_{P_{mn,(l-m)n}^{(r)}({\A})}^{GL_{nl}^{(r)}({\A})}({\mathcal L}_{mn,\chi}^{(r)}
\otimes {\mathcal L}_{(l-m)n,\chi}^{(r)})\delta_{P_{mn,(l-m)n}}^s.$$ 
Then, as explained in Section~\ref{residue} above, ${\mathcal L}_{n,\chi}^{(r)}$ is the residue of 
this Eisenstein series at $s=\frac{r+1}{2r}$. Consider the constant term
$E_{nl,\chi}^{(r),U}(g,s)$ for $\text{Re}(s)$ large. Unfolding this constant term  as
in \cite{B-F-G, J-L,  M-W, O-S}, we obtain a sum of Eisenstein series (and degenerate Eisenstein series), where the sum is over 
Weyl elements that give a complete set of representatives for the double cosets
$P_{mn,(l-m)n}(F)\backslash GL_{nl}(F)/P_{(l-m)n,mn}(F)$. 
(See for example Bump-Friedberg-Ginzburg \cite{B-F-G}, Eq.\ (1.2).) Let
$$w_0=\begin{pmatrix} &I_{mn}\\ I_{(l-m)n}&\end{pmatrix}.$$ 
Then as in the references above, for every Weyl element not equal to $w_0$  which contributes
a nonzero term, the corresponding Eisenstein series is holomorphic at 
$s=\frac{r+1}{2r}$. 
The contribution from $w_0$ is just the intertwining operator $M_{w_0,s}$ which clearly has a
simple pole at $s=\frac{r+1}{2r}$, and as a function of $(v_1,v_2)$ is as in \eqref{type11}.

The claim about the dependence of ${(\varphi_{nl,\chi}^{(r)}})^U(t(v_1,v_2))$ on
$t$ follows since ${\mathcal L}_{n,\chi}^{(r)}$ is a sub-representation
of the induced representation
$$\Ind_{P_{mn,(l-m)n}^{(r)}({\A})}^{GL_{nl}^{(r)}({\A})}({\mathcal L}_{mn,\chi}^{(r)}
\otimes {\mathcal L}_{(l-m)n,\chi}^{(r)})\delta_{P_{mn,(l-m)n}}^s$$
at the point $s=\frac{r-1}{2r}$.
\end{proof}

We next give an application of Proposition~\ref{prop1}.
\begin{lemma}\label{lem0}
The representation ${\mathcal L}_{nl,\chi}^{(r)}$ is square integrable.
\end{lemma}
\begin{proof}
We use Jacquet's criterion as stated in \cite{M-W}, the Lemma in I.4.11.   Note that ${\mathcal L}_{nl,\chi}^{(r)}$
consists of automorphic forms so the Lemma there is applicable.  Let $U$ denote a 
unipotent radical of a maximal parabolic subgroup $P$ of $GL_{nl}$. Let $U_{n,\ldots, n}^-=\widetilde{w}U_{n,\ldots, n}\widetilde{w}^{-1}$. Here $\widetilde{w}$ is the longest Weyl element in $GL_{nl}$. 

Suppose first that $U$ is such that there is no Weyl element $w$ of $GL_k$ such that 
$wUw^{-1}\subset U_{n,\ldots, n}^-$. Then a standard unfolding argument implies that the 
constant term
$$\int\limits_{U(F)\backslash U({\A})}E_{nl,\chi}^{(r)}(ug,{\mathbf s})\,du$$ 
is zero for all choices of data. 

On the other hand, if $U=U_{(l-m)n,mn}$ for some $m$, then it follows from Proposition~\ref{prop1}, part  (i), that for all $t$ in the center of $P=P_{(l-m)n,mn}$ we obtain the exponent 
$\delta_P^{\frac{r-1}{2r}}=\delta_P^{\frac{-1}{2r}}\delta_B^{\frac{1}{2}}$. Here $B$ is the Borel
subgroup of $GL_{ln}$. Lemma~\ref{lem0} follows. 
\end{proof}

The above Proposition and Lemma can be extended to the general case. That is, both
statements holds for the representation ${\mathcal L}_{k,\lambda,\chi}^{(r)}$ as well.

Since we are in the case $r\geq n$, the representation ${\mathcal L}_{n,\chi}^{(r)}={\Theta}_{n,\chi}^{(r)}$ is clearly generic. 
On the other hand 
if $l$ is chosen so that $nl>r$, then ${\mathcal L}_{nl,\chi}^{(r)}$ is not generic. Hence, there
is a minimal natural number, which we denote by $a$, such that ${\mathcal L}_{an,\chi}^{(r)}$ is generic, but ${\mathcal L}_{(a+1)n,\chi}^{(r)}$ is not. Notice that since $n$ does not divide $r$
then $an<r$. Let $b$ be the smallest natural number
so that $abn>r$. 

For the proof of the next Proposition we need to modify
our construction. Let ${\mathcal E}_{an,\chi}^{(r)}$ denote an irreducible generic summand of 
the representation ${\mathcal L}_{an,\chi}^{(r)}$. The existence of such a summand follows from
Lemma~\ref{lem0} and from the assumption that ${\mathcal L}_{an,\chi}^{(r)}$ is generic. 
Forming the Eisenstein series on $GL_{nl}^{(r)}({\A})$ attached to 
$({\mathcal E}_{an,\chi}^{(r)},{\mathcal E}_{an,\chi}^{(r)},\cdots,{\mathcal E}_{an,\chi}^{(r)})\eta_{{\mathbf s}}$
then it follows as in the previous section this series has a simple
pole at the point $\eta_{{\mathbf s}}=\delta_{P_{n,\ldots,n}}^{\frac{r+1}{2r}}$. Denote the residue
representation by ${\mathcal E}_{abn,\chi}^{(r)}$. It is clear from this construction that 
Proposition~\ref{prop1} holds if we replace the representation ${\mathcal L}_{ln,\chi}^{(r)}$
with the representation ${\mathcal E}_{ln,\chi}^{(r)}$.

Given an automorphic representation $\pi$ defined on a reductive group $H({\A})$, let ${\mathcal O}(\pi)$
be its set of unipotent orbits as defined in Ginzburg  \cite{G1}. (For information about
unipotent orbits see Collingwood and McGovern \cite{C-M}.) A unipotent orbit 
${\mathcal O}$ is in the set ${\mathcal O}(\pi)$ if first, $\pi$ has no nonzero Fourier coefficients 
attached to any unipotent orbit which is greater than ${\mathcal O}$ or not comparable with $\mathcal O$ and second, $\pi$ has a nonzero
Fourier coefficient corresponding to the unipotent orbit ${\mathcal O}$. Since unipotent groups split in any
covering group, this definition 
extends without change to representations of metaplectic groups.  Moreover, for the general linear group the unipotent orbits
are parametrized by partitions, a manifestation of the Jordan decomposition.
In our case, we have:
\begin{proposition}\label{propuni1}
We have ${\mathcal O}({\mathcal E}_{abn,\chi}^{(r)})=((an)^b)$.
\end{proposition}
\begin{proof}
This proof is similar to Jiang and Liu \cite{J-L}; see also the proof of Proposition~5.3 in Ginzburg \cite{G1}. 

We need to prove two things. First, let ${\mathcal O}=(n_1n_2\ldots n_r)$  be a partition of $abn$. 
Assume that this partition is greater than or is not related to 
the partition $((an)^b)$. Then we need to prove
that any Fourier coefficient of ${\mathcal O}({\mathcal E}_{abn,\chi}^{(r)})$ associated with this partition is zero. As explained in \cite{G1} at the beginning of the proof of Proposition~5.3,
it is enough to prove that ${\mathcal O}({\mathcal E}_{abn,\chi}^{(r)})$ has no nonzero
Fourier coefficients associated with the partitions $(m1^{abn-m})$ for all $m>an$. 

The proof of this statement about the Fourier coefficients is similar to \cite{G1, J-L}. In \cite{G1}
this was proved by local means, and this was replaced in \cite{J-L} by a version of Proposition~\ref{prop1}.  
To indicate the approach, suppose that $m$ is even. The case that $m$ is odd is similar and
will be omitted. If $m$ is even, then the Fourier coefficient associated with the unipotent
orbit  $(m1^{abn-m})$ is given as follows. Let $P_m$ denote the parabolic subgroup of $GL_{abn}$
whose Levi part is $GL_1^m\times GL_{abn-m}$. We embed the Levi part in $GL_{abn}$ as all matrices
of the form $\text{diag}(a_1,\ldots,a_{m/2},h,b_1,\ldots,b_{m/2})$, with $a_i, b_i\in GL_1$, 
$h\in GL_{abn-m}$. Let $V_m^0$ denote the unipotent radical of $P_m$, and $V_m$ denote the subgroup
of $V_m^0$ which consists of all matrices $v=(v_{i,j})$ such that $v_{i,abn-\frac{m}{2}+1}=0$ for
all $\frac{m}{2}+1\le i\le abn-\frac{m}{2}$. Let $\psi_{V_m}$ denote the character of $V_m$ defined
as follows. For $v=(v_{i,j})\in V_m$ set
$$\psi_{V_m}(v)=\psi\Big(v_{\frac{m}{2},abn-\frac{m}{2}+1}+
\sum_{i=1}^{m/2 -1}(v_{i,i+1}+v_{abn-\frac{m}{2}+i,abn-\frac{m}{2}+i+1})\Big).$$ 
Then, the 
Fourier coefficient corresponding to the partition $(m1^{abn-m})$ is given by 
$$\int\limits_{V_m(F)\backslash V_m({\A})}E_{abn,\chi}^{(r)}(v)\,\psi_{V_m}(v)\,dv.$$

Let $w_1$ denote the Weyl element of $GL_{abn}$ defined by
$$w_1=\begin{pmatrix} I_{\frac{m}{2}}&&\\ &&I_{\frac{m}{2}}\\ &I_{abn-m}&\end{pmatrix}.$$ 
Conjugating
by $w_1$ and performing some Fourier expansions, one deduces that the vanishing of the
above integral for all choices of data is equivalent to the vanishing of 
\begin{equation}\label{zero1}
\int\limits_{U_m(F)\backslash U_m({\A})}E_{abn,\chi}^{(r)}(u)\,\psi_{U_m}(u)\,du
\end{equation}
for all choices of data. Here $U_m$ is the unipotent radical of the standard parabolic subgroup 
of $GL_{abn}$ whose Levi part is $GL_1^{m-1}\times GL_{abn-m+1}$, with the Levi part embedded
in $GL_{abn}$ as  $\text{diag}(a_1,\ldots, a_{m-1},h)$ ($a_i\in GL_1$, $1\leq i\leq m-1$,
and $h\in GL_{abn-m+1}$), 
and $\psi_{U_m}$ is the character
$$\psi_{U_m}(u)=\psi(u_{1,2}+u_{2,3}+\cdots +u_{m-1,m}).$$ 
Note that when $m=abn$, the group $U_{abn}$ is the maximal upper triangular unipotent subgroup of $GL_{abn}$.

Let  $\underline{\alpha}=(\alpha_i)_{m\le i\le abn-1}$ with
$\alpha_i\in\{0,1\}$ for all $i$
and define
$$\psi_{U_m,\underline{\alpha}}(u)=\psi\Big(\sum_{i=1}^{m-1} u_{i,i+1}+
\sum_{i=m}^{abn-1}
\alpha_{i}u_{i,i+1}\Big).$$
Then performing Fourier expansions, one sees that the
vanishing of the integral \eqref{zero1} is equivalent to the vanishing of all the integrals
\begin{equation}\label{zero2}
\int\limits_{U_{abn}(F)\backslash U_{abn}({\A})}E_{abn,\chi}^{(r)}(u)\,
\psi_{U_m,\underline{\alpha}}(u)\,du.
\end{equation}

If $\alpha_i=1$ for all $i$, then the integral \eqref{zero2} is the Whittaker coefficient of $E_{abn,\chi}^{(r)}$
which is zero. If instead $\alpha_i=0$ for some $i$, let
$k\ge m$ be the first integer such that $\alpha_i=1$ for all $m\le i\le k$ and
$\alpha_{k+1}=0$. If $k\ne np$ for some natural number $p$, then the corresponding integral
\eqref{zero2} is zero. Indeed, since ${\Theta}_{n,\chi}^{(r)}$ is a cuspidal representation, 
the constant term  $E_{abn,\chi}^{(r),U}(g,s)$ is zero
if $U$ is not equal to $U_{(ab-l)n,ln}$ for some $l$. 
(Note that it is precisely at this point in the argument that we use the cuspidality hypothesis.)
On the other hand, if $k=np$, then it
follows from Proposition~\ref{prop1} that integral \eqref{zero2} is zero if the residue 
representation ${\mathcal E}_{np,\chi}^{(r)}$ is not generic. But since $np=k\ge m>an$, it
follows from the definition of $a$ that ${\mathcal E}_{np,\chi}^{(r)}$ is indeed not generic.
This completes the proof that ${\mathcal E}_{abn,\chi}^{(r)}$ has no nonzero Fourier coefficient
corresponding to any unipotent orbit which greater than or not related to $((an)^b)$.

The last step is to prove that ${\mathcal E}_{abn,\chi}^{(r)}$ has a nonzero Fourier coefficient
corresponding to the partition $((an)^b)$. This is proved similarly to  \cite{G1}
pp.\ 338-339; see also \cite{J-L} and \cite{O-S}. 
Let $E_{abn,\chi}^{(r)}$ be a vector in the space of ${\mathcal E}_{abn,\chi}^{(r)}$.
Then it follows from \cite{G1}, p.\ 338, that the Fourier coefficient of $E_{abn,\chi}^{(r)}$
with respect to the orbit $((an)^b)$ is given by the integral
\begin{equation}\label{f1}
f(h)=\int\limits_{V(F)\backslash V({\A})}E_{abn,\chi}^{(r)}(vh)\,\psi_V(v)\,dv.
\end{equation}
Here we let the $V_{k,p}$ be the unipotent subgroup of $GL_{kp}$
consisting of all matrices of the from
\begin{equation}\label{mat1}
\begin{pmatrix} I_{k}&X_{1,2}&*&*&\cdots&*\\ &I_{k}&X_{2,3}&*&\cdots&*\\ &&I_{k}&X_{3,4}&\cdots&*\\
&&&I_{k}&\cdots&*\\ &&&&\ddots&*\\ &&&&&I_{k}\end{pmatrix}
\end{equation}
with $I_k$ appearing $p$ times and each $X_{i,j}$ a matrix of size $k$. 
The group $V$ in the integral \eqref{f1} is the group $V_{k,p}$ with $k=b$
and $p=an$. Also, define a character
$\psi_{V_{k,p}}$ on $V_{k,p}$ by $\psi_{V_{k,p}}(v)=\psi(\text{tr} (X_{1,2}+X_{2,3}+\cdots +X_{p-1,p}))$.
 Then the character $\psi_V$ in \eqref{f1} is $\psi_{V_{b,an}}$.
 
Let $U_{an}$ denote the maximal upper unipotent subgroup of $GL_{an}$, and let
$U'=U_{an}\times\cdots\times U_{an}$ where the group $U_{an}$ appears $b$
times. This group is embedded in $GL_{anb}$ as $(u_1,\ldots,u_b)\mapsto 
\text{diag}(u_1,\ldots,u_b)$.  Let $\psi_{U'}$ be the character given by
$$\psi_{U'}(u')=\psi_{U_{an}}(u_1)\ldots \psi_{U_{an}}(u_b),$$
where $\psi_{U_{an}}$ is the standard Whittaker character of $U_{an}$. 
Then as in \cite {G1} p.\ 338,
the integral \eqref{f1} is nonzero for some choice of data if and only if the integral
\begin{equation}\label{f11}
\int\limits_{U'(F)\backslash U'({\A})}\int\limits_{V_{b,an}(F)\backslash  V_{b,an}({\A})}
E_{abn,\chi}^{(r)}(vu')\,\psi_{U'}(u')\,dv\,du'
\end{equation}
is not zero for some choice of data. Using
\eqref{type11} inductively and the irreducibility of the representation ${\mathcal E}_{an,\chi}^{(r)}$, 
we deduce that the integral \eqref{f11} is not zero for some choice of data if
the representation ${\mathcal E}_{an,\chi}^{(r)}$ is generic. This last assertion follows from 
our assumption on the number $a$.
\end{proof}

We can now prove the first part of Theorem~\ref{th1}.

\begin{proposition}\label{prop2}
Let $n\le r$ be a natural number, and suppose there exists a cuspidal theta representation ${\Theta}_{n,\chi}^{(r)}$  on $GL_n^{(r)}({\A})$. Then $n$ divides $r$.
\end{proposition}

\begin{proof}
Suppose instead that $n$ does not divide $r$.
Construct the representation ${\mathcal E}_{abn,\chi}^{(r)}$ on the group $GL_{abn}^{(r)}({\A})$ as above. 
It follows from Proposition~\ref{propuni1} that ${\mathcal O}({\mathcal E}_{abn,\chi}^{(r)})=((an)^b)$.  
Let $E_{abn,\chi}^{(r)}$ be a vector in the space of ${\mathcal E}_{abn,\chi}^{(r)}$.
Then the Fourier coefficient of $E_{abn,\chi}^{(r)}$
with respect to the orbit $((an)^b)$ is given by the integral \eqref{f1} above.

Since $an<r$ and $b\ge 2$, the Fourier coefficient \eqref{f1} defines a genuine 
automorphic function on some covering  group of $GL_b({\A})$ of degree greater than one. 
Hence $f(h)$ cannot be the constant function. 
Note that at this step we are using the hypothesis that $n$ does not divide $r$. Indeed, this assumption implies that 
$an\neq r$. By contrast, if $an=r$, then it might happen that the above embedding of the group $GL_b$ splits
under the $r$-fold cover, and we would not be able to assert that $f(h)$ is not constant.

Let $\sigma$ denote the representation generated by all functions $f(h)$ as above.  Since a non-constant automorphic function cannot equal a constant term along any unipotent subgroup, it follows that the integral
\begin{equation}\label{f2}
\int\limits_{F\backslash {\A}}f(x(l))\,\psi(l)\,dl\qquad\text{where~} x(l)=I_b+le_{1,b}
\end{equation}
is not zero for some function $f$ in $\sigma$ (here and below $e_{i,j}$ denotes the $(i,j)$-th elementary matrix).
Using this nonvanishing, we will show
that the representation ${\mathcal E}_{abn,\chi}^{(r)}$ has a nonzero Fourier coefficient corresponding to the unipotent orbit
$((an+1)(an)^{b-1}(an-1))$.

To do so, we introduce two families of unipotent subgroups of $GL_{abn}$.
First, let $Z_i$, $1\le i\le an-1$, denote the unipotent subgroup with
$$Z_i(\A)=\{r_1e_{b,1}+r_2e_{b,2}+\cdots +r_{b-1}e_{b,b-1}\ \ :\ \ r_j\in {\A}\}
\subset X_{i,i+1}$$
 and let $Z_0$ denote the group with
$$Z_0(\A)=\{r_2e_{2,1}+r_3e_{3,1}+\cdots +r_{b-1}e_{b-1,1}\ \ :\ \ r_j\in {\A}\}
\subset X_{1,2}.$$ 
Here each $X_{i,i+1}$ is embedded in $GL_{abn}$ as in \eqref{mat1}. 
Notice that $Z_0$ and $Z_1$ are two distinct subgroups of $X_{1,2}$.
Second, for $1\le i\le an-1$ let
$$Y_i(\A)=\{I_{b}+l_1e_{1,b}+l_2e_{2,b}+\cdots + l_{b-1}e_{b-1,b}\ \ :\ \ l_j\in {\A}\}$$
and let
$$Y_0(\A)=\{I_{b}+l_2e_{1,2}+l_3e_{1,3}+\cdots + l_{b-1}e_{1,b-1}\ \ :\ \ l_j\in {\A}\}$$
These groups  are embedded in $GL_{abn}$ as $\text{diag}(Y_0,Y_1,\ldots,Y_{an-1})$.
Also, let $Z$ be the unipotent subgroup of $GL_{abn}$ generated by all $Z_i$ with $0\le i\le an-1$,
and let $Y$ be the unipotent subgroup of $GL_{abn}$ generated by the $Y_i$, $0\le i\le an-1$, 
together with the one dimensional unipotent subgroup
$x(l)$ defined in \eqref{f2}.

Substituting \eqref{f2} into \eqref{f1} we then expand the integral along the unipotent subgroups $Y_i$ where $0\le i\le an-1$.
Then using the unipotent subgroups $X_i$, we obtain that
the integral \eqref{f2} is equal to
\begin{equation}\label{f3}
\int\limits_{Z({\A})}\int\limits_{V_1(F)\backslash V_1({\A})}E_{abn,\chi}^{(r)}(v_1zh)\,\psi_{V_1}(v_1)\, dv_1\,dz,
\end{equation}
where $V_1$ is the subgroup of $V$ generated by $Y$ and all the one-parameter unipotent subgroups $\{x_\alpha(t)\}$, $\alpha$ a positive root,
that are in $V$
but not in $Z$. The character $\psi_{V_1}$ matches $\psi_{V}$ on the one-parameter subgroups $\{x_\alpha(t)\}$ in $V$ that are not in $Z$ and is $\psi(y_{1,b})$
on $Y(\A)$.

To conclude the proof, we note that
the inner integration over $V_1$ in  \eqref{f3} is a Fourier coefficient corresponding to the unipotent orbit $((an+1)(an)^{b-1}(an-1))$. Since it is not zero this contradicts Proposition~ \ref{propuni1}.
\end{proof}

\section{The Uniqueness Property}\label{unique}

In this section we prove the uniqueness property given in Theorem~\ref{th1}. To do so, suppose that
there are two natural numbers, $n$ and $m$, $m<n$, with cuspidal theta representations  ${\Theta}_{n,\chi}^{(r)}$ and 
${\Theta}_{m,\chi}^{(r)}$ attached to the same character $\chi$. We shall derive  a contradiction.

As above, let $a\ge 1$ be the smallest natural number such that ${\mathcal E}_{an,\chi}^{(r)}$
is an irreducible  generic representation. From Section~\ref{divide} 
we know that $n$ divides $r$, and hence, using \cite{K-P}, we have $an\le r$. Choose the  smallest
integer $b\ge 1$ such that $abn+m>r$. Construct the residue representation 
${\mathcal E}_{abn+m,\chi}^{(r)}$ as in Section~\ref{residue}. This representation
is the residue of the Eisenstein series on $GL_{nl+m}^{(r)}({\A})$ 
attached to the induced representation 
\begin{equation}\label{ind23}\notag
\Ind_{P_{n,\ldots,n,m}^{(r)}({\A})}^{GL_{nl+m}^{(r)}({\A})}({\mathcal E}_{an,\chi}^{(r)}\otimes {\mathcal E}_{an,\chi}^{(r)}\otimes\cdots\otimes {\mathcal E}_{an,\chi}^{(r)}
\otimes\Theta_{m,\chi}^{(r)})\eta_{{\mathbf s}}
\end{equation}

Denote by $U_{n,\ldots,n,m}^-$ the
transpose of the unipotent group $U_{n,\ldots,n,m}$ defined in Section~\ref{residue}.

The following Lemma is standard.
\begin{lemma}\label{lemuni1}
Let $U$ denote the unipotent radical of a maximal parabolic subgroup of $GL_{abn+m}$. 
If there is no Weyl 
element $w$ in $GL_{abn+m}$ such that $wUw^{-1}$ is a subgroup of $U_{n,\ldots,n,m}^-$, then the
constant term $E_{abn+m,\chi}^{(r),U}(g)$ is zero for all choices of data.
\end{lemma}
With this we have the following analogue of Proposition~\ref{prop1}.
\begin{proposition}\label{prop21}
Let  $U$ denote the unipotent radical of the maximal parabolic subgroup of $GL_{abn+m}$ whose Levi part is $GL_{r_1}\times GL_{r_2}$ with $r_1=m+kn$ and $r_2=(ab-k)n$ for some $k\ge 0$. Suppose that $wUw^{-1}$ is a subgroup of $U_{n,\ldots,n,m}^-$ for some Weyl element $w$. Let $E_{abn+m,\chi}^{(r)}$ be a vector in the space of ${\mathcal E}_{abn+m,\chi}^{(r)}$.
Then for $i=1,2$ there exist $E_{r_i,\chi}^{(r)}$ in
the space of ${\mathcal E}_{r_i,\chi}^{(r)}$ such that
\begin{equation*} 
E_{abn+m,\chi}^{(r),U}\left (\begin{pmatrix} v_1&\\ &v_2\end{pmatrix}\right )=E_{m+kn,\chi}^{(r)}(v_1)
E_{(ab-k)n,\chi}^{(r)}(v_2)
\end{equation*}
for all unipotents $v_i\in GL_{r_i}(\A)$.
Moreover a statement similar to Proposition~\ref{prop1}, part (ii),
holds in this case as well.
\end{proposition}
With these properties we can prove
\begin{proposition}\label{propuni2}
Under the hypotheses of this section, ${\mathcal O}({\mathcal E}_{abn+m,\chi}^{(r)})=((an)^bm)$.
\end{proposition}
\begin{proof}
There are two things to establish.  The first is the vanishing property of the
Fourier coefficients with respect to orbits that are greater than
or incomparable with $((an)^bm)$.
This vanishing is established similarly to the proof of Proposition~\ref{propuni1} above. We omit the details. 

The second part of the assertion is the nonvanishing of a Fourier coefficient attached to the partition
$((an)^bm)$. We now describe such a coefficient.
The description depends on the parity relation between $an$ and $m$. We shall give
the details in the case where both numbers are odd. The other cases are similar.

Let $V$ denote the unipotent subgroup of $GL_{abn+m}$ consisting of all matrices of the form
\begin{equation}\label{mat5}
\begin{pmatrix} v_1&v_4&v_6\\ &v_2&v_5\\ &&v_3\end{pmatrix}\qquad v_1,v_3\in V_{b,(an-m)/2},\quad
v_2\in V_{b+1,m}.\notag
\end{equation}
Here the groups $V_{k,p}$ were defined in \eqref{mat1} above, and $v_4,v_5,v_6$ are general suitably-sized matrices.
Write $v_4=\begin{pmatrix} *&*\\ v_4'&*\end{pmatrix}$ where $v_4'\in \Mat_{b\times b}$, and let 
$\psi_1(v_4)=\psi(\text{tr}\ v_4')$. Similarly, write 
$v_5=\begin{pmatrix} *&*\\ v_5'&*\\ v_5''&*\end{pmatrix}$ with $v_5'\in \Mat_{b\times b}$ and
$v_5''\in \Mat_{1\times b}$, and let $\psi_2(v_5)=\psi(\text{tr}\ v_5')$.
Let $\psi_V$ be the character 
$$\psi_V(v)=\psi_{V_{b,(an-m)/2}}(v_1)\,\psi_{V_{b+1,m}}(v_2)\,\psi_{V_{b,(an-m)/2}}(v_3)\,
\psi_1(v_4)\,\psi_2(v_5).$$ 
Then a
Fourier coefficient associated with the partition $((an)^bm)$ is given by
\begin{equation}\label{f4}
\int\limits_{V(F)\backslash V({\A})}E_{abn+m,\chi}^{(r)}(vh)\,\psi_V(v)\,dv.
\end{equation}

Let $\nu_1$ be the Weyl element of $GL_{abn+m}$ defined as follows. Write
$$\nu_1=\begin{pmatrix} w_0\\ w_1\\ \vdots\\ w_b\end{pmatrix}\qquad w_0\in \Mat_{m\times (anb+m)},\qquad
w_j\in \Mat_{an\times (anb+m)},~ 1\le j\le b.$$
Here the matrix $w_0$ has  $(i, b(i+t))$ entries equal to $1$ , $1\le i\le m$,  
and all other entries $0$, where $t=(an-m)/2$. The matrices $w_j$, $1\le j\le b$, have 
entries of $1$ at the $(i_1,j+(i_1-1)b),\ (t+i_2, j+tb+i_2(b+1))$ and 
$(t+m+i_3+1, j+tb+m(b+1)+i_3b)$ positions for all $1\le i_1\le t+1$, 
$1\le i_2\le m$ and $1\le i_3\le t-1$, and all other entries $0$. 
This Weyl element may be characterized as follows. As explained in \cite{G1}, to any unipotent orbit ${\mathcal O}$ one can attach a one  dimensional torus $\{h_{\mathcal O}(t)\}$. In our case, for the unipotent orbit
${\mathcal O}=((an)^bm)$, 
\begin{equation}\label{tor1}\notag
h_{\mathcal O}(t)=\text{diag}\ (t^{an-1}I_b, t^{an-3}I_b,\ldots,t^{m+1}I_b,t^{m-1}I_{b+1},\ldots,
t^{-(m-1)}I_{b+1}, t^{-(m+1)}I_b,\ldots, t^{-(an-1)}I_b).
\end{equation}
The Weyl element $\nu_1$ is the shortest Weyl element in $GL_{abn+m}$ which conjugates the
torus $\{h_{\mathcal O}(t)\}$ to the torus $\{h(t)\}$ with $h(t)=\text{diag}(d_m(t),d_{an}(t),\ldots, d_{an}(t))$, where for
all $i>0$ we have $d_i(t)=\text{diag}(t^{i-1}, t^{i-3},\ldots,  t^{-(i-3)}, t^{-(i-1)})$.

Using the invariance of $E_{abn+m,\chi}^{(r)}$ by  $\nu_1$ and moving it rightward via conjugation, the integral \eqref{f4} is equal to
\begin{equation}\label{f5}
\int\limits_{Z(F)\backslash Z({\A})}\int\limits_{U'(F)\backslash U'({\A})}
\int\limits_{Y(F)\backslash Y({\A})}E_{abn+m,\chi}^{(r)}(yu'z\nu_1 h)\,\psi_{U'}
(u')\,dy\,du'\,dz.
\end{equation}
The notation here is as follows. Let $U_k$ denote the maximal unipotent subgroup
of $GL_k$ consisting of upper triangular matrices. Then
$U'=U_m\times U_{an}\times\cdots\times U_{an}$ where the group $U_{an}$ appears $b$
times. This group is embedded inside $GL_{anb+m}$ as $(u_0,u_1,\ldots,u_b)\mapsto 
\text{diag}(u_0,u_1,\ldots,u_b)$.  The character $\psi_{U'}$ is given by
$$\psi_{U'}(u')=\psi_{U_m}(u_0)\psi_{U_{an}}(u_1)\ldots \psi_{U_{an}}(u_b),$$
where $\psi_{U_k}$ is the standard Whittaker character of $U_k$. 
The group $Y$ is the upper triangular unipotent group defined by $Y=\nu_1V\nu_1^{-1}\cap U_{m,an,\ldots,an}$. 
The group $Z$ is the  lower triangular unipotent group consisting of all elements $v\in V$ such that 
$\nu_1v\nu_1^{-1}\in U_{m,an,\ldots,an}^-$ where the group $U_{m,an,\ldots,an}^-$ is the transpose of the
unipotent group $U_{m,an,\ldots,an}$. 
Another way of characterizing these groups is by means of
the torus $\{h(t)\}$. The group $Y$ is generated by the matrices $y_{i,j}(k)=I_{abn+m}+ke_{i,j}\in U_{m,an,\ldots,an}$ 
such that $h(t)y_{i,j}(k)h(t)^{-1}=y_{i,j}(t^\ell k)$ for some $\ell>0$. Similarly, the group $Z$ is generated by all matrices $z_{i,j}(k)=I_{abn+m}+ke_{i,j}\in U_{m,an,\ldots,an}^-$ such that $h(t)z_{i,j}(k)h(t)^{-1}=z_{i,j}(t^\ell k)$ for some $\ell>0$.

The next step is to perform certain Fourier expansions on the integral \eqref{f5}, using root exchange and the vanishing
of the Fourier coefficients of the representation 
${\mathcal E}_{abn+m,\chi}^{(r)}$  corresponding to unipotent orbits
which are greater than or not comparable to $((an)^bm)$. This process is fairly standard -- see
for example the proof of Ginzburg-Rallis-Soudry \cite{G-R-S}, Lemma 2.4 --  and so we only sketch the ideas. View $U_{m,an,\ldots,an}$
as the group of matrices generated by $u_{i,j}(k)=I_{abn+m}+ke_{i,j}$. Similarly for the groups $Y$
and $Z$. Consider the subgroup $u_{an+m-1,abn+m}(k)$. Since $h(t)u_{an+m-1,abn+m}(k)h(t)^{-1}=
u_{an+m-1,abn+m}(k)$, this one dimensional  unipotent group is not in $Y$. Similarly, 
conjugating by $h(t)$ we deduce that $u_{i,abn+m}(k)$ is in $Y$ for all $1\le i\le an+m-2$,
and that $z_{abn+m-1,an+m-1}(k)$ is in $Z$.
We may continue this process, going from the 
last to the first column in $U_{m,an,\ldots,an}$. When we encounter a unipotent group of the form
$u_{i,j}(k)$ in $U_{m,an,\ldots,an}$ which is not in $Y$ we look for a suitable unipotent subgroup 
of $Z$. If such a subgroup exists, we perform a root exchange. If not, we check that the Fourier coefficient 
obtained corresponds to a unipotent
orbit which is greater than or not related to $((an)^bm)$. This implies that all non-trivial
characters of the expansion contribute zero, and we are left with only the trivial character.

By this argument, we see that  integral \eqref{f4} is
not zero for some choice of data if and only if the integral
\begin{equation}\label{f6}
\int\limits_{U'(F)\backslash U'({\A})}
\int\limits_{U_{m,an,\ldots,an}(F)\backslash U_{m,an,\ldots,an}({\A})}E_{abn+m,\chi}^{(r)}(uu'h)\,\psi_{U'}(u')\,du\,du'
\end{equation}
is not zero for some choice of data. 

Notice that $U_{m,abn}$ is a subgroup of $U_{m,an,\ldots,an}$, and it is the unipotent radical
of the maximal parabolic subgroup $P_{m,abn}$. Hence we can apply inductively Proposition~\ref{prop21} with
to deduce that the integral \eqref{f6} is not zero for some choice of data if  the
two integrals
\begin{equation}\label{f7}
\int\limits_{U_m(F)\backslash U_m({\A})}
\theta_{m,\chi}^{(r)}(u)\,\psi_{U_m}(u)\,du
\end{equation}
and 
\begin{equation}\label{f8}
\int\limits_{U_{abn}(F)\backslash U_{abn}({\A})}
E_{an,\chi}^{(r)}(u)\,\psi_{U_{an}}(u)\,du
\end{equation}
are each nonzero for suitable data. 
The integral \eqref{f7} is not zero since $\Theta_{m,\chi}^{(r)}$ is an irreducible  cuspidal representation, and hence generic. It follows from the irreducibility of 
${\mathcal E}_{an,\chi}^{(r)}$ and from the definition of $a$, that the second integral, \eqref{f8}, is also nonzero for some choice of data.
\end{proof}

With the above we can now prove
\begin{proposition}\label{prop3}
Fix $r$ and $\chi$. Then there is at most one natural number $n$ such that a cuspidal
theta representation ${\Theta}_{n,\chi}^{(r)}$ exists.
\end{proposition}

\begin{proof}
Recall that we are supposing that $m<n$ and there exist cuspidal theta
representations ${\Theta}_{n,\chi}^{(r)}$ and ${\Theta}_{m,\chi}^{(r)}$. We will derive a contradiction.
Form the residue
representation ${\mathcal E}_{abn+m,\chi}^{(r)}$ as above. Then the integral \eqref{f4}
is not zero for some choice of data. We claim that this implies that the integral
\begin{equation}\label{f9}
\int\limits_{U'(F)\backslash U'({\A})}
\int\limits_{U_{an,\ldots,an,m}(F)\backslash U_{an,\ldots,an,m}({\A})}E_{abn+m,\chi}^{(r)}(uu'h)\,\psi_{U'}(u')\,du\,du'
\end{equation}
is not zero for some choice of data. Notice the difference between the integrals \eqref{f9} and
\eqref{f6}. In \eqref{f6} the integration is over $U_{m,an,\ldots,an}(F)\backslash U_{m, an,\ldots,an}({\A})$, while in
\eqref{f9} it is over
$U_{an,\ldots,an,m}(F)\backslash U_{an,\ldots,an,m}({\A})$. These are two different groups.

To prove that the integral \eqref{f9} is not zero for some choice of data, we start with the integral \eqref{f4},
 which we have already shown is nonzero for some choice of data. Let $\nu_2$ be the Weyl element
$$\nu_2=\begin{pmatrix}  w_1\\ \vdots\\ w_b\\ w_0\end{pmatrix}\qquad w_0\in \Mat_{m\times (anb+m)}, \quad
w_j\in \Mat_{an\times (anb+m)},~ 1\le j\le b,$$
where the matrices $w_i$ are as above. Inserting $\nu_2$ into
\eqref{f4} and performing similar Fourier expansions, we deduce that the integral \eqref{f9}
is not zero for some choice of data. Here $U'=U_{an}\times\cdots\times U_{an}\times U_m$,
and the character $\psi_{U'}$ is defined accordingly.
But notice that $U_{abn,m}$ is a subgroup of 
$U_{an,\ldots,an,m}$ which is also the unipotent radical of a maximal parabolic subgroup. However, there
is no Weyl element $w$ such that $wU_{abn,m}w^{-1}$ is contained in $U_{n,\ldots,n,m}^-$. Hence,
by Lemma~\ref{lemuni1} we obtain that the integral \eqref{f9} is zero for all choices of data. This
is a contradiction.
\end{proof}

\section{The Condition on the Character $\chi$}\label{char}

Suppose that $\Theta_{n,\chi}^{(r)}$ is a cuspidal theta representation and $\chi=\chi_1^r$
for some character $\chi_1$. We shall derive a contradiction. The idea is similar to the
one we used in Section~\ref{unique}. 

First, by Flicker \cite{F}, we know that Theorem~\ref{th1} holds if $n=2$. Let $n\ge 3$.
Similarly to \cite{K-P}, we  may construct the theta representation $\Theta_{2,\chi}^{(r)}$ by means of a residue of an Eisenstein series. This is possible since $\chi=\chi_1^r$. This 
representation has a nonzero constant term and so is not cuspidal. Define $a$ and $b$ as in Section~\ref{unique} above
and let $m$ defined
in that section equal 2 in the present case. Then construct the residue representation 
${\mathcal E}_{abn+m,\chi}^{(r)}={\mathcal E}_{abn+2,\chi}^{(r)}$ as above. 
Although the representation  $\Theta_{2,\chi}^{(r)}$ is not cuspidal, most
of the results stated in Section~ \ref{unique} go through with small adaptations. In particular, we have
\begin{proposition}\label{propuni3}
Under the hypotheses of this section, ${\mathcal O}({\mathcal E}_{abn+m,\chi}^{(r)})=((an)^b2)$.
\end{proposition}
Then we obtain a contradiction as in Section~\ref{unique}. Indeed, from Proposition~\ref{propuni3} 
we obtain that the Fourier coefficient 
\begin{equation*} 
\int\limits_{U'(F)\backslash U'({\A})}
\int\limits_{U_{an,\ldots,an,2}(F)\backslash U_{an,\ldots,an,2}({\A})}E_{abn+2,\chi}^{(r)}(uu'h)\,\psi_{U'}(u')\,du\,du'
\end{equation*}
is not zero for some choice of data. Notice that $U_{abn,2}$ is a subgroup of $U_{an,\ldots,an,2}$.
However, even though $\Theta_{2,\chi}^{(r)}$ is not cuspidal, the constant term of 
$E_{abn+2,\chi}^{(r)}$ along $U_{abn,2}$ is still zero for all choices of data.  Indeed, it is not hard to check that there is no 
Weyl element $w$ in $GL_{abn+2}$ such that $wU_{abn,2}w^{-1}$ is contained in 
$U^-_{an,\ldots,an,2}$. The vanishing of this constant term then follows  as in Lemma~\ref{lemuni1}.
However, this is a 
contradiction, and the result follows.

\end{document}